\documentclass[11pt, a4paper]{amsart}
\usepackage{amssymb}
\usepackage{amsmath}
\usepackage{amssymb}
\usepackage{amsthm}
\usepackage{bbm}
\usepackage{bm}
\usepackage{mathrsfs}
\usepackage[T1]{fontenc}
\usepackage[usenames,dvipsnames,svgnames,table]{xcolor}
\usepackage{esint}
\usepackage{mathtools}

\usepackage{enumitem}
\setenumerate{label={\rm (\alph{*})}}
\usepackage[colorlinks=true,citecolor=black,urlcolor=black,
linkcolor=black]{hyperref}

\usepackage{graphicx}
\usepackage{tikz}

\usepackage{geometry}
\allowdisplaybreaks

\newtheorem{theorem}{Theorem}
\newtheorem{lemma}[theorem]{Lemma}
\newtheorem{proposition}[theorem]{Proposition}

\newtheorem{corollary}[theorem]{Corollary}

\newtheorem{remark}[theorem]{Remark}

\usepackage{booktabs}

\normalsize

\def\XXint#1#2#3{{\setbox0=\hbox{$#1{#2#3}{\int}$ }
\vcenter{\hbox{$#2#3$ }}\kern-.6\wd0}}

\newcommand{\di}{\operatorname{div}}

\newcommand{\dif}{\operatorname{d}\!}

\newcommand{\ymarrow}{\overset{\mathbf{Y}}{\rightarrow}}

\newcommand{\N}{\mathbb{N}}

\newcommand{\R}{\mathbb{R}}
\newcommand{\A}{\mathbb{B}}

\newcommand{\B}{\mathbb{B}}
\newcommand{\cala}{\mathcal{A}}

\newcommand{\locc}{\operatorname{loc}}

\newcommand{\ball}{\operatorname{B}}

\newcommand{\sobo}{\operatorname{W}}

\newcommand{\lebe}{\operatorname{L}}

\newcommand{\hold}{\operatorname{C}}

\newcommand{\curl}{\operatorname{curl}}
\renewcommand{\leq}{\leqslant}

\newcommand{\imag}{\operatorname{i}}
\newcommand{\id}{\operatorname{Id}}

\newcommand{\e}{\operatorname{e}}

\renewcommand{\di}{\operatorname{div}}

\newcommand{\lin}{\mathrm{Lin}}

\renewcommand{\e}{\operatorname{e}}

\newcommand{\rank}{\operatorname{rank\,}}

\usepackage{tikz-cd}
\begin{document}
\title[Potentials for $\mathcal{A}$--quasiconvexity]{Potentials for $\mathcal{A}$--quasiconvexity}
\author[B. Rai\c{t}\u{a}]{Bogdan Rai\c{t}\u{a}}\thanks{\emph{Author's Adress}: Zeeman Building, University of Warwick, Coventry CV4 7HP, United Kingdom; \emph{email}: \texttt{bogdan.raita@warwick.ac.uk}; \emph{phone}: +44 24 765 73420.}

\begin{abstract}
We show that each constant rank operator $\mathcal{A}$ admits an exact potential $\A$ in frequency space. We use this fact to show that the notion of $\mathcal{A}$--quasiconvexity can be tested against compactly supported fields. We also show that $\mathcal{A}$--free Young measures are generated by sequences $\A u_j$, modulo shifts by the barycentre.
\end{abstract}
\maketitle
\section{Introduction}
A challenging question in the study of non--linear partial differential differential equations is to find which non--linear functionals are well--behaved with respect to weak convergence, which represents the typical topology consistent with physical measurements and has satisfactory compactness properties. In the context of the Calculus of Variations, answering this question amounts, roughly speaking, to describing semi--continuity properties of functionals 
\begin{align}\label{eq:functional}
\mathscr{E}[w]=\int_\Omega f(w(x))\dif x
\end{align}
with respect to weak convergence in certain weakly closed, convex subsets $\mathfrak{C}$, say, of $\lebe^p$--spaces, $1<p<\infty)$, under growth conditions
\begin{align}\label{eq:growth_pot}
0\leq f\leq c(|\cdot|^p+1)
\end{align}
on the integrands $f$. Such subsets $\mathfrak{C}$ can account for differential constraints and boundary conditions. Modulo terms removed for simplicity of exposition, such functionals could model, for instance, the energy arising from the deformation of a solid body $\Omega$, viewed as a sufficiently regular open subset of $\R^n$, where $f$ is a continuous energy density map characterized by the constitutive properties of the material. In accordance with the Direct Method in the Calculus of Variations, imposing a suitable bound from below on $f$ ensures existence and weak compactness of minimizing sequences $w_j$. The appropriate continuity property of $\mathscr	E$ in this case is that of lower semi--continuity with respect to weak convergence in $\lebe^p$
\begin{align*}
w_j\rightharpoonup w\implies \liminf_{j\rightarrow\infty}\mathscr{E}[w_j]\geq \mathscr{E}[w],
\end{align*}
which, if satisfied, implies existence of a minimizer $w\in\mathfrak{C}$.

It is well--known that if $\mathfrak{C}$ consists of the whole of $\lebe^p$, then $\mathscr{E}$ is weakly sequentially lower semi--continuous if and only if $f$ satisfying \eqref{eq:growth_pot} is convex. Of course, convexity of $f$ is sufficient for lower semi--continuity (always understood as weakly sequential throughout this note)  in any reasonable class $\mathfrak{C}$, but it is hardly necessary in general. For instance, if $\mathfrak{C}$ is the space of weak gradients in $\lebe^2$ and $f$ is a quadratic form, then one can easily show that $f$ being positive on rank--one matrices implies lower semi--continuity. This example, that we will later come back to in more generality, is of particular relevance, as it provides the insight for a second convexity condition, which is necessary for lower semi--continuity with the constraint $w=\nabla u$: if $\mathscr{E}$ is lower semi--continuous, then $f$ is convex along rank--one lines. In particular, for integrands $f$ of class $\hold^2$, this is equivalent to the so--called Legendre--Hadamard ellipticity condition
\begin{align*}
\frac{\partial^2F(X)}{\partial X_{ij}\partial X_{\alpha\beta}}a_ia_\alpha b_jb_\beta\geq 0\quad\text{for all } X,a,b,
\end{align*}
where summation over repeated indices is adopted. From this point of view, lower semi--continuity of $\mathscr{E}$ acting on gradients reflects a semi--convexity condition on $f$. Indeed, it was shown by \textsc{Morrey} in \cite{Morrey} 
that lower semi--continuity of $\mathscr	E$ is equivalent with \emph{quasiconvexity} of $f$, i.e., the Jensen--type inequality
\begin{align*}
f(\eta)\leq\fint_Q f(\eta+\nabla u(x))\dif x
\end{align*}
holds for all $\eta$ and all smooth maps $u$ with compact support in the open cube $Q$. On one hand, the quasiconvexity assumption is a plausible constitutive relation for energy functionals arising in solid mechanics \cite{Ball77_0}; on the other hand, it is but a minor improvement of the lower semi--continuity concept, which makes it particularly difficult to check in applications. The counterexample of \textsc{\v{S}ver\'ak} \cite{Sv} rules out the possibility of quasiconvexity being a type of directional convexity (see also \cite[Ex.~3.5]{BCO} for the case of higher order gradients). A tractable sufficient condition for quasiconvexity is \emph{polyconvexity}, i.e., $f$ is a convex functions of the minors, also introduced by \textsc{Morrey} in \cite{Morrey} in connection with lower semi--continuity and used by \textsc{Ball} to obtain existence theorems under very mild growth conditions, giving very satisfactory existence results in non--linear elasticity \cite{Ball76}. The fact that quasiconvexity does not imply polyconvexity is much easier to see, at least in higher dimensions, and follows from an old observation of \textsc{Terpstra} concerning quadratic forms \cite{Terpstra} (see also \cite{Ball85,AD} and the references therein).

The above considerations show that a considerable amount of work was devoted to the treatment of lower semi--continuity in the case when $\mathfrak{C}$ consists of gradients (see \cite{AF,Marcel} and the monographs \cite{Da,Pe}).
However, for instance in continuum mechanics, it is often the case that $\mathfrak{C}$ consists of those $\lebe^p$--fields $w$ that satisfy a linear, typically under--determined, partial differential constraint, say $\mathcal{A}w=0$, assumption that we make henceforth. Examples arise in elasticity, plasticity, elasto--plasticity, electromagnetism, and others. The $\mathcal{A}$--free framework originates in the pioneering work of \textsc{Murat} and \textsc{Tartar} in compensated compactness \cite{Murat0,Tartar1,Tartar2} and can be correlated with the question of finding energy functionals that are continuous with respect to weak convergence in $\mathfrak{C}$ \cite{Murat}. The latter question was also studied in generality by \textsc{Ball}, \textsc{Currie}, and \textsc{Olver} in \cite{BCO}, leading to the generalization of polyconvexity to the case where energy functionals depend on higher order derivatives. In this case, the definition of quasiconvexity extends mutatis mutandis \cite{Meyers}. As to the question of lower semi--continuity, the analysis of the case when $f$ is a quadratic form (see, e.g., \cite[Ch.~17]{Tartar3} or \cite[Thm.~2]{Tartar4}) reveals a different necessary condition of directional convexity, namely with respect to the so--called wave cone of $\mathcal{A}$. It was shown by \textsc{Dacorogna} in \cite[Thm.~I.2.3]{Da82} that, in order to have lower semi--continuity, it is sufficient to assume the following generalization of quasiconvexity, namely that
\begin{align*}
f(\eta)\leq\fint_Qf(\eta+w(x))\dif x
\end{align*}
for all $\eta$ and all bounded $w$ such that $\int_Qw=0$ and $\mathcal{A}w=0$. However, it is not clear whether this condition is necessary. More recently, \textsc{Fonseca} and \textsc{M\"uller} showed in \cite{FM99} that if one assumes in addition that the fields $w$ are periodic, in which case $f$ is called \emph{$\mathcal{A}$--quasiconvex}, then one indeed obtains a necessary and sufficient condition\footnote{For comparison, see also \textsc{Seregin}'s work \cite{Se99} in incompressible linearized elasticity, where the methods used to project on solenoidal fields do not require Fourier analysis.} (under suitable growth assumptions on $f$). Their result holds under the assumption that the symbol map $\mathcal{A}(\cdot)$ of $\mathcal{A}$ is a constant rank matrix--valued field away from 0. This condition, introduced in \cite[Def.~1.5]{SW} to prove coerciveness inequalities for non--elliptic systems, was first used in the context of compensated compactness by \textsc{Murat} and ensures, as noted on \cite[p.502]{Murat0}, the continuity of the map
\begin{align}\label{eq:proj_murat}
0\neq\xi\mapsto\text{Proj}_{\ker\mathcal{A}(\xi)},
\end{align}
making tools from pseudo--differential calculus available. In the absence of the constant rank assumption, little is known about the lower semi--continuity problem. One of the few results in this direction was proved by \textsc{M\"uller} in \cite{Muller}, answering a long standing question of \textsc{Tartar} (see also \cite{MuMu} for a generalization).

In the proof of the main result of \cite{FM99}, considerable difficulty is encountered when proving sufficiency of $\mathcal{A}$--quasiconvexity. One reason for this is the absence of potential functions for $\mathcal{A}$, which, if available, should allow one to test with compactly supported functions in the definition of $\mathcal{A}$--quasiconvexity and, perhaps, use more standard methods.

The main result of the present work is to show that the existence of such a potential in Fourier space is equivalent with the constant rank condition.
\begin{theorem}\label{thm:main_pot}
Let $\mathcal{A}$ be a linear, homogeneous differential operator with constant coefficients on $\R^n$. Then $\mathcal{A}$ has constant rank if and only if there exists a linear, homogeneous differential operator $\B$ with constant coefficients on $\R^n$ such that 
\begin{align}\label{eq:exact_pot}
\ker\mathcal{A}(\xi)=\mathrm{im\,}\B(\xi)
\end{align}
for all $\xi\in\R^n\setminus\{0\}$.
\end{theorem}
Here $\mathcal{A}(\cdot)$, $\B(\cdot)$ denote the (tensor--valued) symbol maps of, respectively, $\mathcal{A},\B$. We say that $\mathcal{A}$ has \emph{constant rank} if the map $0\neq\xi\mapsto\rank\mathcal{A}(\xi)$ is constant (see Section \ref{sec:algebra} for detailed notation). We will regard $\B$ as the \emph{potential} and $\mathcal{A}$ as the \emph{annihilator}, although this terminology is not standard.

It is important to mention that the algebraic relation \eqref{eq:exact_pot} does \emph{not}, in general, imply for vector fields $w$ that
\begin{align}\label{eq:exact_fn}
\mathcal{A}w=0\implies w=\B u\qquad\text{ for some }u.
\end{align}
To see this, simply take $\mathcal{A}=\nabla^k$. In turn, if we impose restrictions on $w$ that allow for usage of the Fourier transform, \eqref{eq:exact_fn} can be shown to hold (Lemma \ref{lem:equal_fields}). As a consequence, standard arguments in the Calculus of Variations lead to the fact that a map $f$ is $\mathcal{A}$--quasiconvex if and only if
\begin{align*}
f(\eta)\leq\fint_Qf(\eta+\B u(x))\dif x
\end{align*}
for all $\eta$ and all smooth vector fields $u$ supported in an open cube $Q$ (Corolla--ry~\ref{cor:A-Aqc}). It is also the case that under the constant rank condition, the notions of $\cala$--quasiconvexity \cite[Def.~3.1]{FM99} and \textsc{Dacorogna}'s $\cala$--$\B$--quasconvexity \cite[Eq.~(A.12)]{Da82} coincide. In particular, one can define $\cala$--quasiconvexity via integration over arbitrary domains (Lemma~\ref{lem:Aqc_dom}). As a consequence, the lower semi--continuity properties of functionals \eqref{eq:functional} in the topologies considered in \cite{FM99,ARDPR}, which are natural from the point of view of compensated compactness theory, rely only on the structure of $\B$.

In fact, we will show that the $\mathcal{A}$--quasiconvex relaxation of a continuous integrand can be described in terms of $\B$ only. From this point of view, it is natural to investigate the Young measures generated by sequences satisfying differential constraints \cite[Sec.~4]{FM99}, as they efficiently describe the minimization of energies that are not lower semi--continuous. We recall that the role of parametrized measures for non--convex problems in the Calculus of Variations was first recognized by \textsc{Young} in the pioneering works \cite{Y1,Y2,Y3}. See the monographs \cite{Mu,Pe} for a modern, detailed exposition.

Roughly speaking, for $1<p<\infty$, we consider a sequence $w_j$ converging weakly in $\lebe^p$ which is asymptotically $\mathcal{A}$--free and generates a Young measure $\bm{\nu}$. Technically speaking, it suffices to take $\mathcal{A}w_j$ to be strongly compact in $\sobo^{-k,p}_{\locc}$, where $k$ is the order of $\mathcal{A}$. This is (slightly more general than) the topology considered in \cite[Rk.~4.2(i)]{FM99} and is consistent with the topology considered in compensated compactness (see, e.g., \cite[Thm.~17.3]{Tartar3}, which essentially deals with the case of linear Euler--Lagrange equations). In this setting, we will show that the Young measure $\bm{\nu}$ is generated by a sequence of smooth maps $\B u_j$, modulo a shift by the barycentre.

To sum up, under the constant rank condition on the annihilator $\mathcal{A}$, the objects characterizing the lower semi--continuous relaxation of functionals defined on $\mathcal{A}$--free vector fields (i.e., $\mathcal{A}$--quasiconvex envelopes and $\mathcal{A}$--free Young measures) can be described only in terms of the potential $\B$ constructed in Theorem \ref{thm:main_pot}. From this point of view, it is the author's opinion that the study of functionals
\begin{align*}
\mathscr{E}[w]=\int_\Omega f(x,w(x))\dif x\text{ for }\mathcal{A}w=0\quad\text{ and }\quad\mathscr{F}[u]=\int_\Omega f(x,\B u(x))\dif x
\end{align*}
is essentially dual (\emph{strictly} under the constant rank condition). See also \cite{Da82_A-B_qc} and the Appendix of \cite{Da82}.

Since testing with the appropriate quantity is fundamental in the study of partial differential equations, we hope that the observations made in this work will increase the flexibility of analyzing functionals in either class described above. On the other hand, the functional $\mathscr	F$ seems better suited for incorporating boundary conditions. This will be pursued elsewhere.

This paper is organized as follows: In Section~\ref{sec:algebra} we prove the main Theorem~\ref{thm:main_pot}, in Section~\ref{sec:Aqc} we prove that $\mathcal{A}$--quasiconvexity can be tested with compactly supported fields $w=\B u$ (Corollary~\ref{cor:A-Aqc}), and in Section~\ref{sec:AYM} we prove that $\mathcal{A}$--free Young measures are shifts of Young measures generated by sequences $\B u_j$.
\subsection*{Acknowledgement}
The author is grateful to Jan Kristensen for introducing him to the problem and for offering
insightful comments and helpful suggestions. This work was supported by Engineering and Physical Sciences Research Council Award EP/L015811/1. This project has received funding from the European Research Council (ERC) under the European Union's Horizon 2020 research and innovation programme under grant agreement No 757254 (SINGULARITY).
\section{Proof of Theorem \ref{thm:main_pot}}\label{sec:algebra}
We take a moment to clarify notation. By a $k$--homogeneous, linear differential operator $\mathcal{A}$ on $\R^n$ from $W$ to $X$ we mean
\begin{align}\label{eq:cal_A}
\mathcal{A}w\coloneqq\sum_{|\alpha|=k}\partial^\alpha\mathcal{A}_\alpha w\qquad\text{for }w\colon\R^n\rightarrow W,
\end{align}
where $\mathcal{A}_\alpha\in\lin(W,X)$ for all multi--indices $\alpha$ such that $|\alpha|=k$, for finite dimensional inner product spaces $W,X$. We also define the (Fourier) symbol map
\begin{align*}
\mathcal{A}(\xi)\coloneqq\sum_{|\alpha|=k}\xi^\alpha\mathcal{A}_\alpha \in\lin(W,X)\qquad\text{for }\xi\in\R^n.
\end{align*}
We also recall the condition mentioned above that $\mathcal{A}$ is of \emph{constant rank} if there exists a natural number $r$ such that
\begin{align*}
\mathrm{rank}\mathcal{A}(\xi)=r\qquad\text{for all }\xi\in\R^n\setminus\{0\}.
\end{align*}

As to the resolution of Theorem \ref{thm:main_pot}, we recall the notion of (\emph{Moore--Penrose}) \emph{generalized inverse}, introduced independently in \cite{Moore,B,Penrose}, to which we refer plainly as the \emph{pseudo--inverse}, although the terminology is not standard. For a matrix $M\in\R^{N\times m}$, its pseudo--inverse $M^\dagger$ is the unique ${m\times N}$ matrix defined by the relations
\begin{align*}
MM^\dagger M=M,\quad M^\dagger MM^\dagger=M^\dagger,\quad (MM^\dagger)^*=MM^\dagger,\quad(M^\dagger M)^*=M^\dagger M,
\end{align*}
where $M^*$ denotes the adjoint (transpose) of $M$. Equivalently, the pseudo--inverse is determined by the geometric property that $MM^\dagger$ and $M^\dagger M$ are orthogonal projections onto $\mathrm{im\,}M$ and $(\ker M)^\perp$ respectively. We refer the reader to the monograph \cite{PI} for more detail on generalized inverses.

With these considerations in mind, it is easy to see that the projection map $\mathbb{P}\in\hold^\infty(\R^n\setminus\{0\},\lin(W,W))$ defined in \eqref{eq:proj_murat} can be represented as
\begin{align}\label{eq:proj}
\mathbb{P}(\xi)=\id_W-\mathcal{A}^\dagger(\xi)\mathcal{A}(\xi)\quad\text{for }\xi\in\R^n\setminus\{0\}.
\end{align}
The smoothness of $\mathbb{P}$ is well--known \cite[Prop.~2.7]{FM99}; for a proof using pseudo--inverses see \cite[Sec.~4]{Pro}. By the basic properties of pseudo--inverses, it is easy to see that, with the choice $\B=\mathbb{P}$, we have that \eqref{eq:exact_pot} holds; however, the tensor--valued map $\mathbb{P}$ is $0$--homogeneous, hence not polynomial in general. In particular, $\mathbb{P}$ cannot define a differential operator.

On the other hand, motivated by a similar construction in \cite[Rk.~4.1]{VS}, one can speculate that $\mathbb{P}$ and, in fact, $\mathcal{A}^\dagger(\cdot)$ are rational functions. This is indeed the case, as a consequence of the main result of \textsc{Decell} in \cite{Decell}, building on the fundamental result of \textsc{Penrose} \cite[Thm.~2]{Penrose} and the Cayley--Hamilton Theorem.
\begin{theorem}[{\textsc{Decell} \cite[Thm.~3]{Decell}}]\label{thm:dec}
Let $M\in\R^{N\times m}$ and denote by
\begin{align*}
p(\lambda)\coloneqq (-1)^N\left( a_0\lambda^N+a_1\lambda^{N-1}+\ldots +a_N\right)\quad\text{for }\lambda\in\R
\end{align*}
the characteristic polynomial of $MM^*$, where $a_0=1$. Define 
\begin{align}\label{eq:rank}
r\coloneqq\max\{j\in\N\colon a_j>0\}.
\end{align}
Then, if $r=0$, we have that $M^\dagger=0$; else
\begin{align*}
M^\dagger=-a^{-1}_rM^*\left[a_0(MM^*)^{r-1}+a_1(MM^*)^{r-2}+\ldots +a_{r-1}\id_{N\times N}\right].
\end{align*}
\end{theorem}
\begin{proof}[Proof of Theorem \ref{thm:main_pot} (sufficiency)]
Suppose that $\mathcal{A}$ has constant rank. We put $M\coloneqq\mathcal{A}(\xi)$ in the above Theorem for $\xi\in\R^n\setminus\{0\}$,  and abbreviate $\mathcal{H}(\xi)\coloneqq\mathcal{A}(\xi)\mathcal{A}^*(\xi)$. The first, perhaps most crucial, observation is that $r(\xi)$, as defined by \eqref{eq:rank}, equals the number of non--zero eigen--values of $MM^*$, which equals the number of singular values of $M$. This is, in turn, equal to $\rank M$, which is independent of $\xi$ by the constant rank assumption on $\mathcal A$.

Therefore, if $r(\xi)=r=0$, we have that $\mathcal{A}(\xi)=0_{N\times m}$, $\mathcal{A}^\dagger(\xi)=0_{m\times N}$, so we can simply choose $\B(\xi)=\id_W$, which satisfies \eqref{eq:exact_pot} and gives rise to a linear, 0--homogeneous differential operator. Otherwise, if $r(\xi)=r>0$, we obtain
\begin{align*}
\mathcal{A}^\dagger(\xi)=-a_r(\xi)^{-1}\mathcal{A}^*(\xi)\left[a_0(\xi)\mathcal{H}(\xi)^{r-1}+a_1(\xi)\mathcal{H}(\xi)^{r-2}+\ldots +a_{r-1}(\xi)\id_{X}\right].
\end{align*}
It is easy to see that $\mathcal{H}(\cdot)$ is a tensor--valued polynomial in $\xi$. The scalar fields $a_j$, $j=1\ldots r$, are such that $a_j(\xi)$ is a coefficient of the characteristic polynomial of $\mathcal{H}(\xi)$, hence a linear combination of minors. In particular, $a_j$ are scalar--valued polynomials in $\xi$.

It then follows that, with $\mathbb{P}$ as in \eqref{eq:proj_murat},
\begin{align}\label{eq:bbA}
\B(\xi)\coloneqq a_r(\xi)\mathbb{P}(\xi)
=a_r(\xi)\id_W-a_r(\xi)\mathcal{A}^\dagger(\xi)\mathcal{A}(\xi)\quad\text{for }\xi\in\R^n
\end{align}
defines a tensor--valued polynomial that satisfies \eqref{eq:exact_pot}. In particular, \eqref{eq:bbA} gives rise to a linear differential operator. To check that it is homogeneous, it suffices to see that $a_r(\cdot)$ is a linear combination of minors of the same order of $\mathcal{H}(\cdot)$, which is homogeneous since $\mathcal{A}(\cdot)$ is.
\end{proof}
The necessity of the constant rank condition in Theorem \ref{thm:main_pot} follows from the following Lemma and the Rank--Nullity Theorem.
\begin{lemma}\label{lem:nec_CR_pot}
Let $S\subset\R^n$ be a set of positive Lebesgue measure and $P,Q$ be two matrix--valued polynomials on $\R^n$. Suppose that there exists $s$ such that
\begin{align*}
\rank P(\xi)+\rank Q(\xi)=s\qquad\text{ for }\xi\in S.
\end{align*}
Then both $P$ and $Q$ have constant rank in $S$.
\end{lemma}
\begin{proof}
We abbreviate $R_P\coloneqq \rank P$, $R_Q\coloneqq\rank Q$ and assume for contradiction that $R_P$ is not constant in $S$. Say $R_P(S)=\{r_1,r_1+1\ldots,r_2\}$ for natural numbers $r_1<r_2$. We also write $\mathrm{M}_d$ for the map that has input a matrix and returns (a vector of) all its minors of order $d$. In particular, $\mathrm{M}_dP$, $\mathrm{M}_dQ$ are vector--valued polynomials on $\R^n$. We then have that
\begin{align*}
R_P^{-1}(\{r_1,r_1+1\ldots r_2-1\})\subset\{\xi\in\R^n\colon\mathrm{M}_{r_2}P(\xi)=0\},
\end{align*}
so that either $\mathrm{M}_{r_2}P\equiv0$ (which is not the case by definition of $r_2$) or $R_P^{-1}(\{r_1,r_1+1\ldots r_2-1\})$ is Lebesgue--null\footnote{For an elementary proof of this fact, see \cite{CT}.}. On the other hand,
\begin{align*}
R_P^{-1}(\{r_2\})\cap S&=R_Q^{-1}(\{s-r_2\})\cap S\\
&\subset R_Q^{-1}(\{s-r_2,s-r_2+1,\ldots s-r_1-1\})\\
&\subset \{\xi\in\R^n\colon\mathrm{M}_{s-r_1}Q(\xi)=0\},
\end{align*}
which is Lebesgue-null by the same argument. Since
\begin{align*}
S=[R_P^{-1}(\{r_1,r_1+1,\ldots r_2-1\})\cap S]\cup [R_P^{-1}(\{r_2\})\cap S],
\end{align*}
it follows that $S$ is Lebesgue--null and we arrive at a contradiction.
\end{proof}

It is natural to ask the reversed question, whether a constant rank operator $\B$ admits an exact annihilator $\mathcal{A}$. This is indeed the case, as can be shown by a simple modification of the argument above:
\begin{remark}\label{rk:ann}
Let $\B$ be a linear, homogeneous, differential operator of \emph{constant rank} on $\R^n$ from $V$ to $W$. Then, we can choose $M\coloneqq\B(\xi)$ for $\xi\in\R^n\setminus\{0\}$ in Theorem \ref{thm:dec}, so that
\begin{align*}
\mathcal{A}(\xi)\coloneqq a_r(\xi)\left[\id_W-\B(\xi)\B^\dagger(\xi)\right]\quad\text{for }\xi\in\R^n
\end{align*}
satisfies \eqref{eq:exact_pot} and gives rise to a differential operator. In particular, the formula is consistent with \cite[Eq.~(4.3)]{VS}. This fact can be used to extend the $\lebe^1$--estimates in \cite{VS,BVS} to constant rank operators.
\end{remark}

We conclude the discussion of algebraic properties with two remarks: Firstly, it is quite convenient that the two constructions presented are explicitly computable. On the other hand, performing the computations on simple examples, e.g., involving only $\di$, $\mathrm{grad}$, $\curl$, one easily notices that the operators constructed via our formulas are often overcomplicated. Perhaps more computationally efficient methods, e.g., in the spirit of \cite[Sec.~4.2]{VS} can be developed.
\section{$\mathcal{A}$--quasiconvexity}\label{sec:Aqc}
The relevance of Theorem \ref{thm:main_pot} for analysis can be seen, for instance, from the fact that periodic $\mathcal{A}$--free fields have differential structure:
\begin{lemma}\label{lem:equal_fields}
Let $\mathcal{A}$, $\B$ be linear, homogeneous, differential operators of constant rank with constant coefficients on $\R^n$ from $W$ to $X$, and from $V$ to $W$, respectively. Assume that \eqref{eq:exact_pot} holds. Then
for all $w\in\hold^\infty(\mathbb{T}_n,W)$ such that $\mathcal{A}w=0$ and $\int_{\mathbb{T}_n}w(x)\dif x=0$, there exists $u\in\hold^\infty(\mathbb{T}_n,V)$ such that $w=\B u$. Similarly, for all $w\in\mathscr{S}(\R^n,W)$ such that $\mathcal{A}w=0$, there exists $u\in\mathscr{S}(\R^n,V)$ such that $w=\B u$.
\end{lemma}
Here $\mathbb{T}_n$ denotes the $n$--dimensional torus, identified in an obvious way with (a quotient of) $[0,1]^n$. The Fourier transform is defined as
\begin{align}\label{eq:FT}
\hat{u}(\xi)\coloneqq\int_{\mathbb{T}_n} u(x)\e^{-2\pi\imag x\cdot\xi}\dif x,
\end{align}
for $\xi\in\mathbb{Z}^n$ and $u\in\hold^\infty(\mathbb{T}_n)$. Also, $\mathscr{S}(\R^n)$ denotes the Schwartz class of rapidly decreasing functions on $\R^n$, where the Fourier transform is defined also by \eqref{eq:FT}, with the amendment that the integral is taken over $\R^n$.
\begin{proof}
Let $w\in\hold^\infty(\mathbb{T}_n,W)$ have zero average and satisfy $\mathcal{A}w=0$, so that
\begin{align*}
w(x)=\sum_{\xi\in\mathbb{Z}^n\setminus\{0\}}\hat{w}(\xi)\e^{2\pi\imag x\cdot \xi},
\end{align*}
for $x\in\mathbb{T}_n$, where the coefficients $\hat{w}(\xi)\in\ker\mathcal{A}(\xi)$ decay faster than any polynomial as $|\xi|\rightarrow\infty$. We define
\begin{align*}
u(x)\coloneqq\sum_{\xi\in\mathbb{Z}^n\setminus\{0\}}\B^\dagger(\xi)\hat{w}(\xi)\e^{2\pi\imag x\cdot \xi},
\end{align*}
for $x\in\mathbb{T}_n$, which is smooth by homogeneity of $\B^\dagger(\cdot)$: say $\B$ has order $l$, then $\B^\dagger(\cdot)$ is $(-l)$--homogeneous. We can thus differentiate the sum term by term to obtain
\begin{align*}
\B u(x)&=(2\pi\imag)^l\sum_{\xi\in\mathbb{Z}^n\setminus\{0\}}\B(\xi)\B^\dagger(\xi)\hat{w}(\xi)\e^{2\pi\imag x\cdot \xi}\\
&=(2\pi\imag)^l\sum_{\xi\in\mathbb{Z}^n\setminus\{0\}}\hat{w}(\xi)\e^{2\pi\imag x\cdot \xi}\\
&=(2\pi\imag)^lw(x),
\end{align*}
where the exactness relation \eqref{eq:exact_pot} is used in the second equality, along with the geometric properties of the pseudo--inverse. The proof of the first case is complete.

We give an analogous argument for the case when $w\in\mathscr{S}(\R^n,W)$ is $\mathcal{A}$--free. We have the pointwise relation $\mathcal{A}(\xi)\hat{w}(\xi)=0$, so that \eqref{eq:exact_pot} implies that $w\in\mathrm{im\,}\B(\xi)$ and we can define
\begin{align*}
\hat{u}(\xi)\coloneqq\B^\dagger(\xi)\hat{w}(\xi),
\end{align*}
which satisfies the required properties.
\end{proof}
We conclude this Section by showing that one can test with compactly supported smooth maps in the definition of $\mathcal{A}$--quasiconvexity.
\begin{corollary}\label{cor:A-Aqc}
Let $\mathcal{A},\,\B$ be as in Lemma \ref{lem:equal_fields} and $f\colon W\rightarrow\R$ be Borel measurable and locally bounded. Then
\begin{align*}
Q_\mathcal{A}f(\eta)&\coloneqq\inf\bigg\{\int_{\mathbb{T}_n}f(\eta+w(x))\dif x\colon w\in\hold^\infty(\mathbb{T}_n,W),\mathcal{A}w=0,\int_{\mathbb{T}_n}w(x)\dif x=0\bigg\},\\
Q^\B f(\eta)&\coloneqq\inf\bigg\{\int_{[0,1]^n}f(\eta+\B u(x))\dif x\colon u\in\hold^\infty_c((0,1)^n,V)\bigg\}
\end{align*}
are equal for all $\eta\in W$. Moreover, if $\B$ has order $l$ and $\alpha\in[0,1)$, we have
\begin{align}\label{eq:env_small_norm}
Q_\mathcal{A}f(\eta)=\inf\bigg\{\int_{[0,1]^n}f(\eta+\B u(x))\dif x\colon u\in\hold^\infty_c((0,1)^n,V),\|u\|_{\hold^{l-1,\alpha}}<\varepsilon\bigg\}
\end{align}
for any $\eta\in W$ and $\varepsilon>0$.
\end{corollary}
The proof follows standard arguments; in particular we follow \cite[Prop.~5.13]{Da} and \cite[Thm.~4.2]{KirKri} and include the proof for completeness of the present work.
\begin{proof}
It is obvious that $Q_\mathcal{A}f\leq Q^\B f$. To prove the opposite inequality, let $\varepsilon>0$, $\eta\in W$, and $w$ be a periodic field as in the definition of $Q_\mathcal{A}f(\eta)$. We will construct $v\in\hold^\infty_c((0,1)^n,V)$ such that 
\begin{align}\label{eq:qc}
\int_{[0,1]^n}f(\eta+\B v(x))\dif x\leq \int_{[0,1]^n}f(\eta+w(x))+\varepsilon.
\end{align}
By Lemma \ref{lem:equal_fields}, we have that $w=\B u$ for a periodic field $u\in\hold^\infty(\mathbb{T}_n,V)$. Say, as before, that $\B$ has order $l$ and define $u_N(x)\coloneqq N^{-l}u(Nx)$ for $N$ sufficiently large. This does not change the value of the integral over the cube. Next, let $\delta>0$ be sufficiently small and truncate to obtain $u^\delta_N\coloneqq\rho^\delta u_N$, where $\rho^\delta\in\hold^\infty_c([0,1]^n)$ is such that $\rho^\delta(x)=1$ if $\mathrm{dist}(x,\partial[0,1]^n)>\delta$ and $|\nabla^j\rho^\delta|\leq C\delta^{-j}$ for $j=0\ldots l$ and some constant $C>0$. We impose $\delta N\geq1$ and leave $\delta$ to be determined. It follows, for $c_1\geq1$ depending on $\B$ only, that
\begin{align*}
|\B u^\delta_N|&\leq |\rho^\delta\B u_N|+c_1\sum_{j=1}^{l}|\nabla^{j}\rho^\delta||\nabla^{l-j} u_N|\\
&\leq c_1C\left( \|\B u\|_{\lebe^\infty}+\sum_{j=1}^l(\delta N)^{-j} \|\nabla^{l-j}u\|_{\lebe^\infty}\right)\\
&\leq c_1C\left(\|\B u\|_{\lebe^\infty}+\sum_{j=0}^{l-1}\|\nabla^{j}u\|_{\lebe^\infty}\right)\eqqcolon c_1C\|u\|_{\sobo^{\B,\infty}}.
\end{align*}
Say $f$ is bounded by $M>0$ on $\ball(0,|\eta|+c_1C\|u\|_{\sobo^{\B,\infty}})$. Hence, if we choose $\delta$ such that 
$\mathscr	L^n\left(\{x\in[0,1]^n\colon\mathrm{dist}(x,\partial[0,1]^n)\leq\delta\}\right)\leq M^{-1}\varepsilon$, we obtain
\begin{align*}
\int_{[0,1]^n}f(\eta+\B u^\delta_N(x))\dif x&\leq \int_{\mathrm{dist}(x,\partial[0,1]^n)<\delta}M\dif x+\int_{[0,1]^n}f(\eta+\B u_N(x))\dif x
\\&\leq M\times M^{-1}\varepsilon+\int_{[0,1]^n}f(\eta+w(x))\dif x,
\end{align*}
which implies \eqref{eq:qc} with $v\coloneqq u_N^\delta$. To prove the equality of the two envelopes, we distinguish two cases: If $Q_\mathcal{A}f(\eta)>-\infty$, we can choose $w$ such that
\begin{align*}
\int_{[0,1]^n}f(\eta+w(x))\dif x \leq Q_\mathcal{A}f(\eta)+\varepsilon,
\end{align*}
and we conclude that $Q_\mathcal{A}f(\eta)=Q^\B f(\eta)$ by \eqref{eq:qc} since $\varepsilon>0$ is arbitrary. If $Q_\mathcal{A}f(\eta)=-\infty$, we choose $w$ such that
\begin{align*}
\int_{[0,1]^n}f(\eta+w(x))\dif x\leq -\varepsilon^{-1},
\end{align*}
so that we can conclude by \eqref{eq:qc} that $Q^\B f(\eta)=-\infty$.

To prove \eqref{eq:env_small_norm}, we need only show that the infimum is smaller than the envelope. Firstly, note as above that by replacing $u$ with $u_N(x)=N^{-l}u(Nx)$, where $u$ is extended by periodicity to $\R^n$, the value of the integral does not change. It suffices to choose $N$ large enough so that $u_N$ has small $\hold^{l-1,\alpha}$--norm. Note that for $j=0\ldots l-1$ we have
\begin{align*}
\|\nabla^ju_N\|_{\infty}=N^{j-l}\|\nabla^ju\|_{\infty},
\end{align*}
which can clearly be made arbitrarily small.

Finally, to check the H\"older bound, say that $\{z_i+[0,N^{-1}]^n\}_{i=1}^{N^n}$ is a covering of $[0,1]^n$ by cubes of side--length $N^{-1}$ that can only touch at their boundaries and let $x,y\in[0,1]^n$. If $x,y$ lie in the same cube $z_i+[0,N^{-1}]^n$, we have that
\begin{align*}
|\nabla^{l-1}u_N(x)-\nabla^{l-1}u_N(y)|&=N^{-1}|\nabla^{l-1}u(Nx-z_i)-\nabla^{l-1}u(Ny-z_i)|\\
&\leq \|\nabla^l u\|_{\infty}|x-y|\\
&\leq (\sqrt{n}N^{-1})^{1-\alpha}\|\nabla^l u\|_{\infty}|x-y|^\alpha,
\end{align*}
which can be made small since $1-\alpha>0$. If $x,y$ lie in different cubes, which we label $Q_x,Q_y$. Let $\bar x\in\partial Q_x\cap(x,y)$, $\bar y\in\partial Q_y\cap(x,y)$, so that $|x-y|\geq|x-\bar x|+|y-\bar y|$, $|x-\bar x|,|y-\bar y|\leq\sqrt{n}N^{-1}$, and all derivatives of $u_N$ vanish near $\bar x,\bar y$. Using these facts and the previous step we get
\begin{align*}
|\nabla^{l-1}u_N(x)-\nabla^{l-1}u_N(y)|&\leq|\nabla^{l-1}u_N(x)-\nabla^{l-1}u_N(\bar x)|\\
&+|\nabla^{l-1}u_N(y)-\nabla^{l-1}u_N(\bar y)|\\
&\leq (\sqrt{n}N^{-1})^{1-\alpha}\|\nabla^l u\|_{\infty}\left(|x-\bar x|^\alpha+|y-\bar y|^\alpha\right)\\
&\leq (\sqrt{n}N^{-1})^{1-\alpha}\|\nabla^l u\|_{\infty}2^{-\alpha}|x-y|^\alpha,
\end{align*}
where the last inequality follows by concavity and monotonicity of $0\leq t\mapsto t^\alpha$. The proof is complete.
\end{proof}
\begin{remark}
\normalfont{Using the argument in Corollary~\ref{cor:A-Aqc}, one can show for constant rank operators $\cala$ that $\cala$--quasiconvexity, as defined by \textsc{Fonseca} and \textsc{M\"uller} in \cite[Def.~3.1]{FM99}, coincides with $\cala$--$\B$--quasiconvexity, as introduced by \textsc{Dacorogna} in \cite{Da82,Da82_A-B_qc} (to be precise, in the original definition of $\cala$--$\B$--quasiconvexity, the operator $\B$ is assumed to be of first order, but this is only a minor technical restriction). In this case, it is not difficult to prove that \cite[Thm.~4]{Da82_A-B_qc} is essentially unconditional. A proof of this fact will be given elsewhere.}
\end{remark}
We also have that $\cala$--quasiconvexity can be defined by integrals over arbitrary domains, instead of cubes.
\begin{lemma}\label{lem:Aqc_dom}
Let $\mathcal{A},\,\B$ be as in Lemma \ref{lem:equal_fields} and $f\colon W\rightarrow\R$ be Borel measurable, locally bounded, and $\cala$--quasiconvex, and $\Omega$ be a bounded open set. Then
\begin{align*}
f(\eta)\leq \fint_\Omega f(\eta+\B v(y))\dif y
\end{align*}
for all $\eta\in W$ and $v\in\hold^\infty_c(\Omega,V)$.
\end{lemma}
The proof follows from a simple argument in the Calculus of Variations \cite[Prop.~5.11]{Da}.
\begin{proof}
Fix $\eta\in W$, $v\in\hold^\infty_c(\Omega,V)$, extended by zero to $\R^n$. By the argument in the proof of Corollary~\ref{cor:A-Aqc}, we write $C\coloneqq(0,1)^n$  and have that
\begin{align*}
f(\eta)\leq \int_C f(\eta +\B u(x) )\dif x
\end{align*}
for all $u\in\hold^\infty_c(C,V)$. For sufficiently small $\varepsilon>0$, we can find $x_0\in\R^n$ such that $x_0+\varepsilon\Omega\subset C$. We define
\begin{align*}
u(x)\coloneqq\varepsilon^l v\left(\dfrac{x-x_0}{\varepsilon}\right),
\end{align*}
so that
\begin{align*}
f(\eta)&\leq \int_C f(\eta+\B u(x))\dif x=|C\setminus(x_0+\varepsilon\Omega)|f(\eta)+\int_{x_0+\varepsilon\Omega}f(\eta+\B u(x))\dif x\\
&=(1-\varepsilon^n|\Omega|)f(\eta)+\int_\Omega f(\eta+\B v(y))\varepsilon^n\dif y.
\end{align*}
Rearranging the terms we obtain the conclusion.
\end{proof}
\section{$\mathcal{A}$--free Young measures}\label{sec:AYM}
We recall the definition of oscillation Young measures, while also giving a simplified variant of the Fundamental Theorem of Young measures.
\begin{theorem}[FTYM, {\cite{Mu,Pe}}]\label{thm:FTYM}
Let $\Omega\subset\R^n$ be a bounded, open set and $z_j\in\lebe^1(\Omega,\R^d)$ be a bounded sequence in $\lebe^1$. Then there exists a subsequence (not relabeled) and a weakly--* measurable map $\bm{\nu}\colon\Omega\rightarrow \mathcal{P}({\R^d})$ (or \emph{parametrized measure} $\bm{\nu}=(\nu_x)_{x\in\Omega}$) such that for all $f\in\hold(\Omega\times\R^d)$ we have that
\begin{align*}
\liminf_{j\rightarrow\infty}\int_\Omega f(x,z_j(x))\dif x\geq \int_{\Omega}\langle f(x,\cdot\,),\nu_x\rangle\dif x
\end{align*}
Moreover,
\begin{align*}
\lim_{j\rightarrow\infty}\int_\Omega f(x,z_j(x))\dif x= \int_{\Omega}\langle f(x,\cdot\,),\nu_x\rangle\dif x
\end{align*}
if and only if the sequence $f(\,\cdot,z_j)$ is uniformly integrable.
\end{theorem}
Above, $\mathcal{P}(\R^d)$ denotes the space of probability measures on $\R^d$. In the notation of Theorem~\ref{thm:FTYM}, we say that \emph{$z_j$ generates the Young measure $\bm{\nu}$} (in symbols, $z_j\ymarrow\bm{\nu}$). We also recall that a sequence $z_j$ is said to be uniformly integrable if and only if for all $\varepsilon>0$, there exists $\delta>0$ such that for all borel sets $E\subset\Omega$, we have that 
\begin{align*}
\mathscr{L}^n(E)<\delta\implies\sup_j\int_E|z_j|\dif x<\varepsilon,
\end{align*}
or, equivalently, if
\begin{align*}
\lim_{\alpha\rightarrow\infty}\sup_j\int_{\{|z_j|>\alpha\}}|z_j|\dif x=0.
\end{align*}
If $|z_j|^p$ is uniformly integrable, we say that $z_j$ is \emph{$p$--uniformly integrable}.
\begin{lemma}[{\cite[Prop.~2.4]{FM99}}]\label{lem:shifts}
Let $z_j$ generate a Young measure $\bm{\nu}$ and $\tilde{z}_j\rightarrow\tilde{z}$ in measure. Then $z_j+\tilde{z}_j$ generates the Young measure $\bm{\mu}$ given by $\mu_x=\nu_x\star\delta_{\tilde{z}(x)}$ for $\mathscr{L}^n$ a.e. $x$, i.e.,
\begin{align*}
\langle\varphi,\mu_x\rangle=\langle \varphi(\,\cdot+\tilde{z}(x),\nu_x\rangle
\end{align*}
for any $\varphi\in\hold_0$.
\end{lemma}
The following is an extension of \cite[Lem.~2.15]{FM99}. The first two steps of the present proof are almost a repetition of their arguments, which we include since the original proof only covers first order annihilators $\mathcal{A}$.
\begin{proposition}\label{prop:AYM_p}
Let $\mathcal{A}$, $\B$ be as in Lemma \ref{lem:equal_fields} and have orders $k$, $l$, respectively, $\Omega\subset\R^n$ be a bounded Lipschitz domain, and $1<p<\infty$. Let $w_j,w\in\lebe^p(\Omega,W)$ be such that
\begin{align*}
w_j\rightharpoonup w&\text{ in }\lebe^p(\Omega,W),\\
\mathcal{A}w_j\rightarrow\mathcal{A}w&\text{ in }\sobo^{-k,p}_{\locc}(\Omega,X),\\
w_j\overset{\mathbf{Y}}{\rightarrow}\bm{\nu}.
\end{align*}
Then there exists a sequence $u_j\in\hold^\infty_c({\Omega},V)$ such that
\begin{align*}
\B u_j\rightharpoonup0&\text{ in }\lebe^p(\Omega,W),\\
\B u_j+w\overset{\mathbf{Y}}{\rightarrow}\bm{\nu}.
\end{align*}
Moreover, $u_j$ can be chosen such that $(\B u_j)_j$ is $p$--uniformly integrable.
\end{proposition}
A Young measure $\bm{\nu}$ satisfying the assumptions of Proposition~\ref{prop:AYM_p} is said to be an \emph{$\mathcal{A}$--free Young measure}.
\begin{proof}
By Lemma~\ref{lem:shifts} and linearity we can assume that $w=0$. We will identify maps defined on $\Omega$ with their extensions by zero to full--space without mention. Uniform integrability considerations strictly refer to sequences defined on $\Omega$.

\emph{Step I}. We construct $p$--uniformy integrable $\tilde w_j\in\hold^\infty_c(\Omega,W)$ such that $\tilde{w}_j\rightharpoonup0$ in $\lebe^p(\Omega,W)$, $\mathcal{A}\tilde w_j\rightarrow 0$ in $\sobo^{-k,q}(\R^n,X)$ for some $1<q<p$, and $\tilde w_j$ generates $\bm{\nu}$.

Recall the truncation operators, defined for $\alpha>0$ by
\begin{align*}
\tau_\alpha A\coloneqq
\begin{cases}
A&\text{ if }|A|\leq\alpha\\
\alpha A/|A|&\text{ if }|A|>\alpha,
\end{cases}
\end{align*}
which are clearly Carath\'eodory integrands. By Theorem~\ref{thm:FTYM}, we have that
\begin{align*}
\lim_{\alpha\rightarrow\infty}\lim_{j\rightarrow\infty}\int_\Omega |\tau_\alpha w_j|^p\dif x&=\lim_{\alpha\rightarrow\infty}\int_{\Omega}\int_{W}|\tau_\alpha A|^p\dif\nu_x(A)\dif x\\
&=\int_{\Omega}\int_{W}|A|^p\dif\nu_x(A)\dif x<\infty,
\end{align*}
so that we can choose a diagonal subsequence $\alpha_j\uparrow\infty$ such that $\int_\Omega |\tau_{\alpha_j}w_j|^p\dif x$ equals the $p$--th moment of $\bm{\nu}$. It also follows from Theorem \ref{thm:FTYM} that $(\tau_{\alpha_j}w_j)_j$ is $p$--uniformly integrable.

We now show that $\tau_{\alpha_j}w_j$ generates $\bm{\nu}$. Since $w_j$ converges weakly in $\lebe^p(\Omega,W)$, it converges weakly in $\lebe^1$, hence is uniformly integrable, so that $\tau_{\alpha_j}w_j-w_j\rightarrow 0$ in measure. It also follows by elementary manipulations that $\tau_{\alpha_j}w_j-w_j\rightharpoonup 0$ in $\lebe^p$, so that, indeed, $\tau_{\alpha_j}w_j$ generates $\bm{\nu}$ by Lemma~\ref{lem:shifts}.

Let $1<q<p$. We have that
\begin{align*}
\|\tau_{\alpha_j}w_j-w_j\|_{\lebe^q(\Omega,W)}\leq\int_{\{|w_j|>\alpha_j\}}2^q|w_j|^q\dif x\leq 2^q\alpha_j^{q-p}\int_{\{|w_j|>\alpha_j\}}|w_j|^p\dif x\rightarrow0,
\end{align*}
so that $\mathcal{A}\tau_{\alpha_j}w_j\rightarrow0$ in $\sobo^{-k,q}_{\locc}(\Omega,X)$. We also record that $\tau_{\alpha_j}w_j$ is precompact in $\sobo^{-1,q}(\Omega,W)$, so that $D^\beta\tau_{\alpha_j} w_j\rightarrow0$ in $\sobo^{-k,q}(\Omega,X)$ for $|\beta|<k$.

We can therefore choose a sequence of cut--off functions $\rho_j\in\hold^\infty_c(\Omega,[0,1])$ such that $\rho_j\uparrow1$ in $\Omega$ and
$\|\rho_j \mathcal{A}\tau_{\alpha_j}w_j\|_{\sobo^{-k,q}(\R^n,X)}\rightarrow0$ and
\begin{align*}
\mathcal{A}(\rho_j\tau_{\alpha_j}w_j)=\rho_j\mathcal{A}\tau_{\alpha_j}w_j+\sum_{m=1}^k B_m[D^m\rho_j,D^{k-m}\tau_{\alpha_j}w_j]\rightarrow0\quad\text{ in }\sobo^{-k,q}(\R^n,X),
\end{align*}
where $B_m$ are fixed bi--linear pairings given by the Leibniz rule. To see that this is possible, consider $\Omega_j\coloneqq\{x\in\Omega\colon\mathrm{dist}(x,\partial\Omega)<j\}$, where $s_j\downarrow0$ will be determined. We require that $\rho_j=1$ in $\Omega\setminus\Omega_{s_j}$, $\rho_j=0$ in $\Omega_{2s_{j}}$ and $|D^{m}\rho_j|\leq cs_j^{-m}$, $m=1,\ldots,k$. It is easy to see that the sum above is controlled in $\sobo^{-k,q}$ by
\begin{align*}
\sum_{m=1}^k\|D^m\rho_j\|_{\lebe^\infty}\|D^{k-m}\tau_{\alpha_j}w_j\|_{\sobo^{-k,q}}\leq c\sum_{m=1}^ks_j^{-m}\|D^{k-m}\tau_{\alpha_j}w_j\|_{\sobo^{-k,q}},
\end{align*}
so that it suffices to choose any $s_j\geq \max_{m=1,\ldots,k}\|D^{k-m}\tau_{\alpha_j}w_j\|_{\sobo^{-k,q}}^{1/(2m)}\downarrow0$ as $j\rightarrow\infty$. Alternatively, one can consider a different cut--off sequence $\rho_i\uparrow1$ and employ a diagonalization argument. 

We define 
\begin{align*}
\tilde w_j\coloneqq(\rho_j\tau_{\alpha_j}w_j)\star\eta_{\varepsilon(j)},
\end{align*}
where $\eta_{\varepsilon(j)}$ denotes a standard sequence of (radial, positive) mollifiers and $\varepsilon(j)\downarrow0$ is such that $\tilde{w}_j\in\hold^\infty_c(\Omega,W)$ and, therefore, $\mathcal{A}\tilde{w}_j\rightarrow0$ in $\sobo^{-k,q}(\R^n,X)$. The latter inequality follows since, for all $\varphi\in\hold^\infty_c(\R^n,W)$ with $\|\varphi\|_{\sobo^{k,q}}\leq 1$,
\begin{align*}
\langle\mathcal{A}\tilde w_j,\varphi\rangle&=\langle\mathcal{A}(\rho_j\tau_{\alpha_j}w_j),\varphi\star\eta_{\varepsilon(j)}\rangle\leq\|\mathcal{A}(\rho_j\tau_{\alpha_j}w_j)\|_{\sobo^{-k,q}}\|\varphi\star\eta_{\varepsilon(j)}\|_{\sobo^{k,q}}\\
&\leq\| \mathcal{A}(\rho_j\tau_{\alpha_j}w_j)\|_{\sobo^{-k,q}}\rightarrow0.
\end{align*}
It is also clear that $\|\tilde w_j-\tau_{\alpha_j}w_j\|_{\lebe^p}\rightarrow0$, so that $\tilde w_j$ is $p$--uniformly integrable, converges weakly to $0$ in $\lebe^p$, and generates $\bm{\nu}$.

\emph{Step II}. We project $\tilde{w}_j$ on the kernel of $\mathcal{A}$ in $\R^n$ and show that $\mathbb{P}\tilde{w}_j$ are $p$--uniformly integrable in $\Omega$, converge weakly to zero in $\lebe^p$, and generate $\bm{\nu}$. Here the $\lebe^2$--orthogonal projection operator $\mathbb{P}$ is given by the multiplier in \eqref{eq:proj},
\begin{align*}
\widehat{\mathbb{P}w}(\xi)\coloneqq\mathbb{P}(\xi)\hat{w}(\xi)=[\id_W-\mathcal{A}^\dagger(\xi)\mathcal{A}(\xi)]\hat{w}(\xi)\quad\text{ for }w\in\mathscr{S}(\R^n,W).
\end{align*}
Since the symbol $\mathbb{P}(\cdot)$ is homogeneous of degree zero, $\mathbb{P}$ is a singular integral operator of convolution type; in particular $\mathbb{P}$ maps Schwartz functions to Schwartz functions. Moreover, we have that
\begin{align*}
\mathscr{F}\left(\tilde{w}_j-\mathbb{P}\tilde{w}_j\right)(\xi)=\B^\dagger(\xi)\B(\xi)\mathscr{F}\tilde{w}_j(\xi)=\mathcal{A}^\dagger\left(\dfrac{\xi}{|\xi|}\right)\dfrac{\widehat{\mathcal{A}\tilde{w}_j}(\xi)}{|\xi|^k}, 
\end{align*}
so that, by boundedness of singular integrals on $\lebe^q$
\begin{align*}
\|\tilde{w}_j-\mathbb{P}\tilde{w}_j\|_{\lebe^q(\R^n,W)}\leq c\left\|\mathscr{F}^{-1}\left(\frac{\widehat{\mathcal{A}\tilde w_j}}{|\cdot|^k}\right)\right\|_{\lebe^q(\R^n,X)}=c\|\mathcal{A}\tilde{w}_j\|_{\sobo^{-k,q}(\R^n,X)}\rightarrow0.
\end{align*}
It immediately follows by Lemma~\ref{lem:shifts} that $\mathbb{P}\tilde{w}_j$ generates $\bm{\nu}$. 
To see that $\mathbb{P}\tilde{w}_j\rightharpoonup0$ in $\lebe^p(\Omega,W)$, we note that, since $\mathbb{P}$ is (pointwisely) self--adjoint, we have, for any $g\in\lebe^{p/(p-1)}(\Omega,W)$,
\begin{align*}
\int_\Omega\langle g,\mathbb{P}\tilde{w}_j\rangle\dif x =\int_\Omega\langle \mathbb{P}g,\tilde{w}_j\rangle\dif x\rightarrow0,
\end{align*}
since $\mathbb{P}g\in\lebe^{p/(p-1)}(\Omega,W)$ by boundedness of singular integrals.

To see that $\mathbb{P}\tilde{w}_j$ is $p$--uniformly integrable, we use the idea in \cite[Lem.~2.14.(iv)]{FM99}. We first note, by boundedness of $\mathbb{P}$ on $\lebe^p$, that
\begin{align*}
\sup_j\|\mathbb{P}\tilde{w}_j-\mathbb{P}\tau_\alpha \tilde{w}_j\|_{\lebe^p(\R^n,W)}\leq c\sup_j\|\tilde{w}_j-\tau_\alpha \tilde{w}_j\|_{\lebe^p(\R^n,W)}\rightarrow0\quad\text{ as }\alpha\rightarrow\infty
\end{align*}
by $p$--uniform integrability of $\tilde{w}_j$. Note that for each fixed $\alpha$, $\mathbb{P}\tau_\alpha \tilde{w}_j$ is bounded in $\lebe^r$ for any $p<r<\infty$, hence is $p$--uniformly integrable. Let $\varepsilon>0$. We choose $\alpha>0$ such that
\begin{align*}
\sup_j\|\mathbb{P}\tilde{w}_j-\mathbb{P}\tau_{\alpha} \tilde{w}_j\|_{\lebe^p(\R^n,W)}<\varepsilon
\end{align*}
and also choose $\delta>0$ such that for each Borel set $E\subset \Omega$ with $\mathscr{L}^n(\Omega)<\delta$, we have that $\int_E|\mathbb{P}\tau_\alpha \tilde{w}_j|^p\dif x<\varepsilon$ for all $j$. It follows that for all such $E$,
\begin{align*}
\int_E |\mathbb{P}\tilde{w}_j|^p\dif x\leq 2^{p-1}\left(\sup_j\int_E |\mathbb{P}\tilde{w}_j-\mathbb{P}\tau_\alpha\tilde{w}_j|^p\dif x+\sup_j\int_E |\mathbb{P}\tau_\alpha\tilde{w}_j|^p\dif x\right)<(2\varepsilon)^p,
\end{align*}
where the right hand side is independent of $j$. The second step is concluded.

\emph{Step III}. Using Lemma \ref{lem:equal_fields}, we can write $\mathbb{P}\tilde{w}_j=\B u_j$, where
$
\hat{u}_j(\xi)\coloneqq \B^\dagger(\xi)\widehat{\mathbb{P}\tilde{w}_j}(\xi)
$, so that $u_j\in\mathscr{S}(\R^n,V)$.  It remains to cut--off $u_j$ suitably. 

Since $\B$ has order $l$, we first note that
\begin{align*}
\widehat{D^lu}(\xi)=\B^\dagger(\xi)\widehat{\B u}(\xi)\otimes\xi^{\otimes l},
\end{align*}
so that $\B u\mapsto D^{l}u$ is a singular integral operator of convolution type. It follows that $D^l u_j$ is bounded in $\lebe^p(\R^n)$ (recall here that $\B u_j=\mathbb{P}\tilde w_j$ is bounded in $\lebe^p$ as $\tilde w_j\in\hold^\infty_c(\Omega,W)$ is a weakly convergent sequence), so $u_j$ is bounded in $\sobo^{l,p}(\Omega,V)$.

By compactness of the embedding $\sobo^{l,p}(\Omega)\hookrightarrow\sobo^{l-1,p}(\Omega)$, we have that $u_j\rightarrow u$ in $\sobo^{l-1,p}(\Omega,V)$. Since $\B u_j\rightharpoonup0$, we have that $\B u=0$. On the other hand, $u=\mathscr{F}^{-1}[\B^\dagger(\cdot)]\star (\B u)=0$, so that $D^{l-m}u_j\rightarrow0$ in $\lebe^p(\Omega)$ for $m=1,\ldots,l$.

We now proceed similarly to Step I. Let $\rho\in\hold^\infty_c(\R^n)$ be such that $\rho_j=1$ in $\Omega\setminus\Omega_{s_j}$ and $|D^{m}\rho_j|\leq cs_j^{-m}$, $m=1,\ldots,l$, where
\begin{align*}
s_j\coloneqq \max_{m=1,\ldots,l}\|D^{l-m}u_j\|_{\lebe^p(\Omega)}^{1/(2m)}\rightarrow0.
\end{align*}
We can then estimate
\begin{align*}
\|\B u_j-\B (\rho_j u_j)\|_{\lebe^p(\Omega)}&\leq\|(1-\rho_j)\B u_j\|_{\lebe^p(\Omega)}+\sum_{m=1}^l\|B_m[D^m\rho_j,D^{l-m}u_j]\|_{\lebe^p(\Omega)}\\
&\leq\|\B u_j\|_{\lebe^p(\Omega_{s_j})}+c\sum_{m=1}^ls^{-m}_j\|D^{l-m}u_j\|_{\lebe^p(\Omega)},
\end{align*}
which tends to zero by $p$--uniform integrability of $\B u_j$ and the choice of $s_j$. Here $B_m$ is another collection of bi--linear pairings given by the product rule. It then follows that $\B(\rho_ju_j)$ converges weakly to zero in $\lebe^p(\Omega,W)$, is $p$--uniformly integrable, and generates $\bm{\nu}$. The proof is complete.
\end{proof}

\end{document}